\documentclass[a4paper,10pt]{article}
\usepackage[utf8x]{inputenc}
\usepackage{fullpage}
\usepackage{amsmath, amsthm}
\usepackage{amsfonts}
\usepackage{graphicx}
\usepackage{epstopdf}
\usepackage{subfigure}
\usepackage{float}

\newtheorem{theorem}{Theorem}[section]
\newtheorem{lemma}[theorem]{Lemma}

\newtheorem{corollary}[theorem]{Corollary}

\newtheorem{remark}[theorem]{Remark}

\title{A contribution to the connections between Fibonacci Numbers and Matrix Theory}
\author{
Miriam Farber \thanks{Department of Mathematics, Technion-IIT,
Haifa 32000, Israel (miriamf@techunix.technion.ac.il, \newline
berman@techunix.technion.ac.il).}
\and
Abraham Berman \footnotemark[1]}
\begin{document}
\maketitle
\begin{abstract}
We present a lovely connection between the Fibonacci numbers and the sums of inverses of $(0,1)-$ triangular matrices, namely, a number $S$ is the sum of the entries of the inverse of an $n \times n$ $(n \geq 3)$ $(0,1)-$ triangular matrix iff $S$ is an integer between $2-F_{n-1}$ and $2+F_{n-1}$. Corollaries include Fibonacci identities and a Fibonacci type result on determinants of family of (1,2)-matrices.

\end{abstract}

\section{\textbf{Introduction}}
One of the ways to motivate students' interest in linear algebra is to present interesting connections between matrices and the Fibonacci numbers \begin{center}
$F_1=F_2=1$ ; $F_n=F_{n-1}+F_{n-2}$ $ (n \geq 3)$.
\end{center}
For example, one can prove that $F_n^2-F_{n-1}F_{n+1}=(-1)^{n+1}$ by using induction and the fact that\\
\\
$\det\left(
                                                                                \begin{array}{cc}
                                                                                  F_n & F_{n-1} \\
                                                                                  F_{n+1} & F_n \\
                                                                                \end{array}
                                                                              \right)=\det\left(
                                                                                \begin{array}{cc}
                                                                                  F_n & F_{n-1} \\
                                                                                  F_{n+1}-F_n & F_n-F_{n-1} \\
                                                                                \end{array}
                                                                              \right)=\det\left(
                                                                                \begin{array}{cc}
                                                                                  F_n & F_{n-1} \\
                                                                                  F_{n-1} & F_{n-2} \\
                                                                                \end{array}
                                                                              \right)=-\det\left(
                                                                                \begin{array}{cc}
                                                                                  F_{n-1} & F_{n} \\
                                                                                  F_{n-2} & F_{n-1} \\
                                                                                \end{array}
                                                                              \right)$.\\
\\
Similarly, one can determine the exact value of the $n^{th}$ Fibonacci number, by calculating the eigenvalues and the eigenvectors of $\left(
                                                                                                                                         \begin{array}{cc}
                                                                                                                                            1& 1\\
                                                                                                                                            1& 0\\
                                                                                                                                         \end{array}
                                                                                                                                       \right)$ and using the equation:\\
\begin{center}
$\left(
   \begin{array}{c}
     F_n \\
     F_{n-1} \\
   \end{array}
 \right)=\left(
           \begin{array}{cc}
             1 & 1 \\
             1 & 0 \\
           \end{array}
         \right)\left(
                  \begin{array}{c}
                    F_{n-1} \\
                    F_{n-2} \\
                  \end{array}
                \right)=\ldots=\left(
                                 \begin{array}{cc}
                                   1 & 1 \\
                                   1 & 0 \\
                                 \end{array}
                               \right)^{n-2}\left(
                                              \begin{array}{c}
                                                1 \\
                                                1 \\
                                              \end{array}
                                            \right).$
\end{center}
As another example of connections between Fibonacci numbers and Matrix Theory, consider lower triangular matrices of the form\\
\begin{center}
 $\left(
                                                                            \begin{array}{cccccc}
                                                                              1 & 0 & \cdots & \cdots & \cdots & 0 \\
                                                                              -1 & 1 & 0 & \cdots & \cdots & 0 \\
                                                                              -1 & -1 & 1 & 0 & \cdots & 0 \\
                                                                              0 & -1 & -1 & 1 & \ddots & \vdots \\
                                                                              \vdots & \ddots & \ddots & \ddots & \ddots & \vdots \\
                                                                              0 & \cdots & 0 & -1 & -1 & 1 \\
                                                                            \end{array}
                                                                          \right).$
\end{center}
The inverses of these matrices are of the form
\\
\begin{center}
 $\left(
                                                                            \begin{array}{ccccccc}
                                                                              1 & 0 & \cdots & \cdots & \cdots & \cdots & 0 \\
                                                                              1 & 1 & 0 & \cdots & \cdots & \cdots & 0 \\
                                                                              2 & 1 & 1 & 0 & \cdots & \cdots & 0 \\
                                                                              3 & 2 & 1 & 1 & 0& \ddots & \vdots \\
                                                                              5 & 3 & 2 & 1 & \ddots & \ddots & \vdots \\
                                                                              \vdots & \ddots & \ddots & \ddots & \ddots & \ddots & \vdots \\
                                                                              \vdots & \cdots &5 & 3 & 2 & 1 & 1 \\
                                                                            \end{array}
                                                                          \right).$
\end{center}
which, due to their remarkable structure, are known as Fibonacci matrices. Various properties of these matrices and their generalizations have been studied, e.g. \cite{Lee, Kim, Wang}.\\
Another interesting connection is given in \cite{Ching_Li}, where it is shown that the maximal determinant of an $n\times n$ (0,1)-Hessenberg matrix is $F_n$.\\
\\
Let $S(X)$ denote the sum of the entries of a matrix $X$. In \cite{Bit-Shun}, Huang, Tam and Wu show, among other results, that a number $S$ is equal to $S(A^{-1})$ for an adjacency matrix (a symmetric $(0,1)-$ matrix with trace zero) $A$ iff $S$ is rational. More generally, they ask what can be said about the sum of the entries of the inverse of a $(0,1)-$ matrix. We consider the class of triangular matrices and show that a number $S$ is equal to $S(A^{-1})$ for a triangular $(0,1)-$ matrix $A$ iff $S$ is an integer. This follows from our main result which shows that for $n \geq 3$, a number $S$ is equal to $S(A^{-1})$ for an $n \times n$ triangular $(0,1)-$ matrix $A$ iff
\begin{center}
$2-F_{n-1} \leq S \leq 2+F_{n-1}$.\\
\end{center}

We use the following definitions and notations.\\
$e$ denotes a vector of ones (so $S(A)=e^TAe$).\\
$A_n$ denotes the set of $n \times n$ invertible $(0,1)-$ upper triangular matrices.\\
We will say that a matrix $A \in A_n$, $n \geq 3$ is \emph{maximizing} if $S(A^{-1})=2+F_{n-1}$ and \emph{minimizing} if $S(A^{-1})=2-F_{n-1}$, and refer to maximizing and minimizing matrices as \emph{extremal matrices}.\\
For a set of vectors $V \subset R^{n}$, a vector $v \in V$ is \emph{absolutely dominant} if for every $u \in V$, $|v_i| \geq |u_i|,\\ i=1,2,\ldots,n$.\\

We will use the following well known properties of Fibonacci numbers (see for example \cite{Nicolai}):
\begin{lemma}\label{lemm:fiball}
\begin{enumerate}
  \item $1+\sum_{k=1}^{n} F_{k}=F_{n+2}$
  \item $1+\sum_{k=1}^{n} F_{2k}=F_{2n+1}$
  \item $\sum_{k=1}^{n} F_{2k-1}=F_{2n}$
\end{enumerate}
\end{lemma}
The main result of the paper is proved in the next section. In the third section we describe a construction of extremal matrices with a beautiful Fibonacci pattern of their inverses, and use it to obtain several Fibonacci identities. The paper is concluded with a Fibonacci type result on determinants of (1,2)-matrices, which is in spirit of the result in \cite{Ching_Li}.

\section{\textbf{The main result}}
\begin{theorem}\label{thm:mainresult}
Let $n \geq 3$. Then $S=S(A^{-1})$, for some $A \in A_n$ iff $S$ is an integer between $2-F_{n-1}$ and $2+F_{n-1}$;
$2-F_{n-1} \leq S \leq 2+F_{n-1}$.
\end{theorem}
\begin{proof}
Obviously, $S(A^{-1})$ must be an integer since $A^{-1}=\frac{adj(A)}{\det(A)}$ and $\det(A)=1$.\\
The main part of the proof consists of
\begin{description}
  \item[(a)] $\max_{A \in A_n} S(A^{-1})=2+F_{n-1}$.
  \item[(b)] $\min_{A \in A_n} S(A^{-1})=2-F_{n-1}$.
  \item[(c)] For every integer $S$ between $2-F_{n-1}$ and $2+F_{n-1}$ there exists $A \in A_n$ such that $S(A^{-1})=S$.
\end{description}
To show (a) and (b) we prove
\begin{lemma}\label{lemm:dominantvector}
Let $V= \big\{e^TA^{-1} | A \in A_n \big\}$. Let us denote, only for the purpose of this lemma, $F_0 \equiv -1$ (this is not a Fibonacci number). Then $v=(v_i)$; $v_i=(-1)^iF_{i-1}$

is an absolutely dominant vector of $V$.
\end{lemma}
\begin{proof}
For $n=1$,  $V=\big\{\left(
                     \begin{array}{c}
                       1 \\
                     \end{array}
                   \right)\big\}$, for $n=2$ , $V=\big\{\left(
                                                              \begin{array}{cc}
                                                                1 & 1 \\
                                                              \end{array}
                                                            \right), \left(
                                                                       \begin{array}{cc}
                                                                         1 & 0 \\
                                                                       \end{array}
                                                                     \right)\big\}$, and for $n=3$ ,\\ $V=\big\{\left(
                                                                                                              \begin{array}{ccc}
                                                                                                                1 & 1 & 1 \\
                                                                                                              \end{array}
                                                                                                            \right),\left(
                                                                                                                      \begin{array}{ccc}
                                                                                                                        1 & 0 & 1 \\
                                                                                                                      \end{array}
                                                                                                                    \right),\left(
                                                                                                                              \begin{array}{ccc}
                                                                                                                                1 & 1 & 0 \\
                                                                                                                              \end{array}
                                                                                                                            \right),\left(
                                                                                                                                      \begin{array}{ccc}
                                                                                                                                        1 & 0 & 0 \\
                                                                                                                                      \end{array}
                                                                                                                                    \right),\left(
                                                                                                                                              \begin{array}{ccc}
                                                                                                                                                1 & 1 &-1 \\
                                                                                                                                              \end{array}
                                                                                                                                            \right)
                                                                     \big\}.$  Therefore the statement holds for $n=1,2,3$. In order to prove the lemma for $n \geq 4$ we will use induction. Suppose the assumption is true for $k<n$. We will show that it is true for $k=n$.

We are going to show now that the vector $v$ that is defined in the lemma is an absolutely dominant vector of the set $V=\big\{e^TA^{-1} | A \in A_n \big\}$. Let $A \in A_n$, then $A$ is of the form $\left(
                                            \begin{array}{ccc}
                                              C & \alpha & \beta \\
                                              0 & 1 & x \\
                                              0 & 0 & 1 \\
                                            \end{array}
                                          \right)
$ where $C \in A_{n-2}$, $\alpha,\beta \in \big\{0,1\big\}^{n-2}$, $x \in \big\{0,1\big\}$ , and therefore
\begin{center}
$A^{-1}=\left(
                                                                                                      \begin{array}{cc}
                                                                                                        C^{-1} & -C^{-1}\left(
                                                                                                                      \begin{array}{cc}
                                                                                                                            \alpha &\beta \\
                                                                                                                          \end{array}
                                                                                                                        \right)\left(
                                                                                                                                 \begin{array}{cc}
                                                                                                                                   1 &-x \\
                                                                                                                                   0 &1 \\
                                                                                                                                 \end{array}
                                                                                                                               \right)\\
                                                                                                        0 & \left(
                                                                                                              \begin{array}{cc}
                                                                                                                1 & -x \\
                                                                                                                0 & 1 \\
                                                                                                              \end{array}
                                                                                                            \right)
                                                                                                         \\
                                                                                                      \end{array}
                                                                                                    \right)=\left(
                                                                                                      \begin{array}{cc}
                                                                                                        C^{-1} & -C^{-1}\left(
                                                                                                                      \begin{array}{cc}
                                                                                                                            \alpha &\beta-x\alpha \\
                                                                                                                          \end{array}
                                                                                                                        \right)\\
                                                                                                        0 & \left(
                                                                                                              \begin{array}{cc}
                                                                                                                1 & -x \\
                                                                                                                0 & 1 \\
                                                                                                              \end{array}
                                                                                                            \right)
                                                                                                         \\
                                                                                                      \end{array}
                                                                                                    \right)$
\end{center}

We will use the following notations:\\
\begin{center}
$e^TC^{-1}=\left(
                      \begin{array}{cccc}
                        c_1 & c_2 & \ldots & c_{n-2} \\
                      \end{array}
                    \right)$\\
\end{center}
\begin{center}
$\alpha=\left(
                      \begin{array}{cccc}
                        \alpha_1 & \alpha_2 & \ldots & \alpha_{n-2} \\
                      \end{array}
                    \right)^T$\\
\end{center}

\begin{center}
$\beta=\left(
                      \begin{array}{cccc}
                        \beta_1 & \beta_2 & \ldots & \beta_{n-2} \\
                      \end{array}
                    \right)^T.$
\end{center}

So
\begin{center}
$e^TA^{-1}=\left(
                      \begin{array}{cccccc}
                        c_1 & c_2 & \ldots & c_{n-2} & 1-\sum_{i=1}^{n-2} \alpha_ic_i & 1-x-\sum_{i=1}^{n-2}c_i(\beta_i-x\alpha_i)\\
                      \end{array}
                    \right).$
\end{center}

Consider the $n^{th}$ entry of $e^TA^{-1}$. Since $c_1=1$, $n \geq 4 $, and $ -1 \leq \beta_i-x\alpha_i \leq 1 $ for all $ 1 \leq i \leq n-2 $,  it is easy to see that
\begin{center}
  $-\sum_{i=1}^{n-2}|c_i| \leq 1-x-\sum_{i=1}^{n-2}c_i(\beta_i-x\alpha_i) \leq \sum_{i=1}^{n-2}|c_i|$
\end{center}
for all possible $x,\alpha_i,\beta_i \in \big\{0,1\big\}, 1 \leq i \leq n-2$. Since $\beta_i-\alpha_i \in \big\{-1,0,1\big\}$ , it is possible to achieve equality in each inequality by taking
\begin{equation}
x=1  \hspace{5 mm} and \hspace{5 mm} \textbf{sign}(\beta_i-\alpha_i)=\textbf{sign}(c_i) ,\hspace{5 mm} 1 \leq i \leq n-2
\end{equation}
or
\begin{equation}
x=1 \hspace{5 mm}  and\hspace{5 mm} \textbf{sign}(\beta_i-\alpha_i)=-\textbf{sign}(c_i) ,\hspace{5 mm} 1 \leq i \leq n-2
\end{equation}
respectively. Now, since $|-\sum_{i=1}^{n-2}|c_i||=|\sum_{i=1}^{n-2}|c_i||$ we get that if $A \in A_n$ is a matrix for which $e^TA^{-1}$ is an absolutely  dominant vector, its $n^{th}$ entry must be equal to either \begin{equation}
-\sum_{i=1}^{n-2}|c_i|
\end{equation}
or
\begin{equation}
\sum_{i=1}^{n-2}|c_i|.
\end{equation}
Note that the maximal (minimal) value of (3) ((4)) is obtained by taking $C$ such that $e^TC^{-1}$ is an absolutely dominant vector of the set $V=\big\{e^TA^{-1} | A \in A_{n-2} \big\}$ (and all the absolutely dominant vectors will give the same value). By the inductive hypothesis and using Lemma ~\ref{lemm:fiball}, the maximal value of (4) is $\sum_{i=1}^{n-2}|c_i|=1+\sum_{i=1}^{n-3}F_i=F_{n-1}$ (and this value may be achieved by choosing an appropriate $C$). Similarly, the minimal value of (3) is $-F_{n-1}$. Let us now consider the $(n-1)^{th}$ entry of $e^TA^{-1}$.
By the inductive hypothesis, the absolute value of the $(n-1)^{th}$ entry of $e^TA^{-1}$ is bounded from above by $F_{n-2}$. By taking $C \in A_{n-2}$ such that $e^TC^{-1}$ is an absolutely dominant vector, choosing $\alpha,\beta$ such that either (1) or (2) is satisfies and using Lemma ~\ref{lemm:fiball} and the inductive hypothesis, we get that the $(n-1)^{th}$ entry of $e^TA^{-1}$ is equal to either:
\begin{equation}
1-\sum_{i=1}^{n-2} \alpha_ic_i=1-\sum_{k=1}^{\lfloor{\frac{n-3}{2}\rfloor}} c_{2k+1}=1+\sum_{k=1}^{\lfloor{\frac{n-3}{2}\rfloor}} F_{2k}=F_{2\lfloor{\frac{n-3}{2}\rfloor}+1}
\end{equation}
or
\begin{equation}
1-\sum_{i=1}^{n-2} \alpha_ic_i=1-c_1-\sum_{k=1}^{\lfloor{\frac{n-2}{2}\rfloor}} c_{2k}=-\sum_{k=1}^{\lfloor{\frac{n-2}{2}\rfloor}} F_{2k-1}=-F_{2\lfloor{\frac{n-2}{2}\rfloor}}
\end{equation}
respectively.\\
Note that if $n$ is odd then expression (5) is equal to $F_{n-2}$, and if $n$ is even then expression (6) is equal to $-F_{n-2}$. In sum, using the inductive hypothesis, we showed that the largest possible absolute value of the $n^{th}$ entry of $e^TA^{-1}$ (such that $A \in A_n $) is $F_{n-1}$. In this case, we showed that it is possible to choose $\alpha$ such that the absolute value of the $(n-1)^{th}$ entry of $e^TA^{-1}$ is $F_{n-2}$. This value is the largest possible absolute value due to the inductive hypothesis. Therefore, we showed that the vector $v$, defined in the lemma, is an absolutely dominant vector for $V=\big\{e^TA^{-1} | A \in A_n \big\}$, and the proof is complete.
\end{proof}
We are now ready to prove (a) and (b).
We represent $A \in A_n$ in the same form as in Lemma ~\ref{lemm:dominantvector}.\\
\begin{eqnarray}
    e^TA^{-1}e \nonumber
    & =  e^T\left(
                                                                                                      \begin{array}{cc}
                                                                                                        C^{-1} & -C^{-1}\left(
                                                                                                                      \begin{array}{cc}
                                                                                                                            \alpha &\beta-x\alpha \\
                                                                                                                          \end{array}
                                                                                                                        \right)\\
                                                                                                        0 & \left(
                                                                                                              \begin{array}{cc}
                                                                                                                1 & -x \\
                                                                                                                0 & 1 \\
                                                                                                              \end{array}
                                                                                                            \right)
                                                                                                         \\
                                                                                                      \end{array}
                                                                                                    \right)e \nonumber \\
    & = 2-x+e^TC^{-1}e-e^TC^{-1}\big(\beta+(1-x)\alpha \big) \nonumber \\
    & = 2-x+e^TC^{-1}\Big(e-\alpha-\beta+x\alpha\Big) \nonumber \\
     \nonumber
\end{eqnarray}
Denote $u=\Big(e-\alpha-\beta+x\alpha\Big)$. Note that if $x=1$ then $u \in \big\{0,1\big\}^{n-2}$, and if $x=0$ then $u \in \big\{-1,0,1\big\}^{n-2}$. In addition, note that
\begin{equation}
max \big\{ 2-x+e^TC^{-1}u| x=0, \alpha,\beta \in \big\{0,1\big\}^{n-2}\big\} \geq max \big\{ 2-x+e^TC^{-1}u| x=1, \alpha,\beta \in \big\{0,1\big\}^{n-2}\big\}
\end{equation}
Now, since $C \in A_{n-2}$, the first entry of $e^TC^{-1}$ is 1. If $x=0$, then in order to minimize the value of $e^TC^{-1}u$  we have to take the first entries of $\alpha$ and $\beta$ to be one. On the other hand, if $x=1$, then in order to minimize the value of $e^TC^{-1}u$  we have to take the first entries of $\beta$ to be one. The difference between these two cases is 1, and therefore
\begin{equation}
min \big\{ 2-x+e^TC^{-1}u| x=0, \alpha,\beta \in \big\{0,1\big\}^{n-2}\big\} \leq min \big\{ 2-x+e^TC^{-1}u| x=1, \alpha,\beta \in \big\{0,1\big\}^{n-2}\big\}
\end{equation}
Since we are only interested in the minimal and the maximal values of $e^TA^{-1}e$ we may assume ,by (7) and (8), that $x=0$. Therefore, $e^TA^{-1}e=2+e^TC^{-1}\Big(e-\alpha-\beta\Big)$. Using the notation of Lemma ~\ref{lemm:dominantvector} we get:\\
                     \begin{equation}
                     min \big\{ 2+e^TC^{-1}\Big(e-\alpha-\beta\Big)|\alpha,\beta \in \big\{0,1\big\}^{n-2}\big\}=2-\sum_{i=1}^{n-2} |c_i|
                     \end{equation}
and
\begin{equation}
max \big\{ 2+e^TC^{-1}\Big(e-\alpha-\beta\Big)|\alpha,\beta \in \big\{0,1\big\}^{n-2}\big\}=2+\sum_{i=1}^{n-2} |c_i|
\end{equation}
Therefore, the minimal and the maximal value of $e^TA^{-1}e$ is achieved by taking $C$ such that $e^TC^{-1}$ is an absolutely dominant vector of $\{e^TA^{-1} | A \in A_{n-2} \}$. Hence, by Lemmas ~\ref{lemm:dominantvector} and ~\ref{lemm:fiball},\\
$\max_{A \in A_n} S(A^{-1})=\max\big\{ 2+e^TC^{-1}\Big(e-\alpha-\beta\Big)|\alpha,\beta \in \big\{0,1\big\}^{n-2}, C \in A_{n-2}\big\}=3+\sum_{i=1}^{n-3} F_i=2+F_{n-1},$\\
and similarly, $\min_{A \in A_n} S(A^{-1})=1-\sum_{i=1}^{n-3} F_i=2-F_{n-1}$.\\
\\
It is well known that every natural number is the sum of distinct Fibonacci numbers. For the proof of (c) we need a slightly stronger observation.
\begin{lemma}\label{lemm:lastlemma}
Let $M$ be a natural number. Let $n$ be an integer for which $F_{n-1} \leq M <F_{n}$. Then $M$ can be represented as a sum of distinct Fibonacci elements from the set $\big\{F_1,F_2, \ldots, F_{n-2}\big\}$.
\end{lemma}
\begin{proof}
The proof is by induction. For $M=1$ the statement is true. Now assume that it is true for all the numbers which are smaller than $M$. We will show that it is true for $M$. Let $n$ be an integer for which $F_{n-1} \leq M <F_{n}$. Since $M <F_{n}$ we get that $M <F_{n-2}+F_{n-1}$, and hence $M-F_{n-2}<F_{n-1}$. Therefore, there exists $n-1 \geq k >0$ such that $F_{k-1} \leq M-F_{n-2}<F_k$, and hence by the inductive hypothesis, $M-F_{n-2}$ can be represented as a sum of distinct Fibonacci elements from the set $\big\{F_1,F_2, \ldots, F_{k-2}\big\}$. Since $n-1 \geq k$, we have $n-3 \geq k-2$, and so $M$ can be represented as a sum of distinct Fibonacci elements from the set $\big\{F_1,F_2, \ldots, F_{n-2}\big\}$.
\end{proof}
We conclude the proof of Theorem ~\ref{thm:mainresult} by proving (c).
Let $S=2+T$; $-F_{n-1} \leq T \leq F_{n-1}$. The cases $T=F_{n-1}$ and $T=-F_{n-1}$ were proved in (a) and (b). For $T=0$, let $A$ be a triangular Toeplitz matrix with first row $\left(
                                            \begin{array}{ccccccc}
                                              1 & 0 & 1 & 0 & 0 & \ldots & 0 \\
                                            \end{array}
                                          \right).$ Then $S(A^{-1})=2$. Similarly, it is easy to prove the claim for any $S$ between 1 and $n$. For the other numbers in $[2-F_{n-1},2+F_{n-1}]$ (and also for $1,2,\ldots,n$), let us consider the expression in (10). It is easy to see that in fact by choosing appropriate $\alpha$ and $\beta$ (and $C$ such that $e^TC^{-1}$ is an absolutely dominant vector), $e^TC^{-1}\Big(e-\alpha-\beta\Big)$ can achieve any value of the form $\alpha_1+\sum_{i=2}^{n-2} \alpha_iF_{i-1}$ where $\alpha_i \in \big\{0,1\big\}$ for all $1 \leq i \leq n-2$. Note that by Lemma ~\ref{lemm:lastlemma}, there exists appropriate set $\big\{\alpha_i\big\}_{i=1}^{n-2}$ such that $T=\sum_{i=2}^{n-2} \alpha_iF_{i-1}$ (we may choose $\alpha_1=0$). Hence, for this choice of $C$, $\alpha$ and $\beta$ we get $A$ such that $S=T+2=e^TA^{-1}e$. We may obtain similar result for the case $S=2-T$  where $0 \leq T \leq F_{n-1}$ by looking at expression (9), and this completes the proof.
\end{proof}
As an analogy to the result on rational numbers of \cite{Bit-Shun} mentioned in the introduction, we now have
\begin{corollary}
A number $S$ is equal to $S(A^{-1})$ for a $(0,1)-$ triangular matrix $A$ iff $S$ is an integer.
\end{corollary}

Define $G_n$ to be the set of $n \times n$ matrices of the form $I+B$ where $B$ is an $n \times n$ upper triangular nilpotent matrix with entries from the interval $[0,1]$. Then, using the fact that for an invertible matrix $A$,  $A^{-1}=\frac{adj(A)}{\det(A)}$, and that for $A \in G_n$, $\det(A)=1$, we have $A^{-1}=adj(A)$ for $A \in G_n$. Thus, since $S(A^{-1})$ is linear in each one of the entries in such matrix $A$, we conclude the following:
\begin{corollary}

$\max_{A \in G_n} S(A^{-1})=2+F_{n-1}$,
$\min_{A \in G_n} S(A^{-1})=2-F_{n-1}$.
\end{corollary}

\begin{remark} \rm
For general $n \times n$ invertible $(0,1)-$ matrix $A$ (which is not not necessarily triangular), the question regarding the minimal or the maximal value that $S(A^{-1})$ may obtain is still open. For $n=3,4,5,6$, the extremal values are exactly the same as in the triangular case. However, for $n=7$, there exist an $n \times n$ invertible $(0,1)-$ matrices $M$ and $N$ (which are presented below)  such that $S(M^{-1})=-7$ and $S(N^{-1})=11$, whereas in the triangular case, the minimal and the maximal values are -6 and 10 respectively.\\
\\
$ M=\left(
  \begin{array}{ccccccc}
    1 & 0 & 1 & 0 & 1 & 0 & 0 \\
    0 & 1 & 1 & 0 & 1 & 0 & 0 \\
    0 & 0 & 1 & 1 & 1 & 1 & 1 \\
    0 & 0 & 0 & 1 & 1 & 0 & 0 \\
    0 & 0 & 0 & 0 & 1 & 1 & 1 \\
    0 & 0 & 1 & 0 & 0 & 1 & 0 \\
    0 & 0 & 1 & 0 & 0 & 0 & 1 \\
  \end{array}
\right) $, $ N= \left(
                  \begin{array}{ccccccc}
                    1 & 0 & 1 & 0 & 1 & 1 & 1 \\
                    0 & 1 & 1 & 0 & 1 & 1 & 1 \\
                    0 & 0 & 1 & 1 & 0 & 0 & 1 \\
                    0 & 0 & 0 & 1 & 1 & 1 & 1 \\
                    0 & 0 & 0 & 0 & 1 & 0 & 0 \\
                    0 & 0 & 0 & 0 & 0 & 1 & 0 \\
                    1 & 1 & 0 & 0 & 0 & 0 & 1 \\
                  \end{array}
                \right)$\\
\\
For larger values of $n$, the difference between the general and the triangular case gets bigger.
\end{remark}
\section{\textbf{Extremal matrices}}
Recall that an invertible triangular $n \times n$ $(0,1)-$matrix $A$ is extremal if $e^TA^{-1}e=2 \pm F_{n-1}$. $I_3$ and $I_4$ are maximizing matrices. The matrices $\left(
                                          \begin{array}{ccc}
                                            1 & 1 & 1 \\
                                            0 & 1 & 0 \\
                                            0 & 0 & 1 \\
                                          \end{array}
                                        \right)$ and $\left(
                                                     \begin{array}{cccc}
                                                       1 & 0 & 1 & 1 \\
                                                       0 & 1 & 1 & 1 \\
                                                       0 & 0 & 1 & 0 \\
                                                       0 & 0 & 0 & 1 \\
                                                     \end{array}
                                                   \right)
                                        $ are minimizing matrices.\\
Following the proof of the main theorem we can construct extremal matrices for $n \geq 5$, that have a beautiful Fibonacci pattern of their inverses. For $l=2,3$, partition the off-diagonal entries of an upper triangular $n \times n$ matrix into $n-l$ sets, $S_0,S_1,\ldots, S_{n-l-1}$. $S_{n-l-1}$ consists of the entries in the first 2 rows of the last $l$ columns. For $i=1,2,\ldots,n-l-2$, $S_i$ consists of the entries immediately to the left or immediately below the entries in $S_{i+1}$. $S_0$ consists of all the remaining entries which are above the main diagonal (2 if $l=2$ and 4 if $l=3$). For example, in the case that $n=9$, Figure ~\ref{f:subfig-1} present the partition in the case $l=2$, and Figure ~\ref{f:subfig-2} present the partition in the case $l=3$.

\begin{figure}[H]
\begin{center}
\mbox{
	\leavevmode
	\subfigure []
	{ \label{f:subfig-1}
	  \includegraphics[width=4cm]{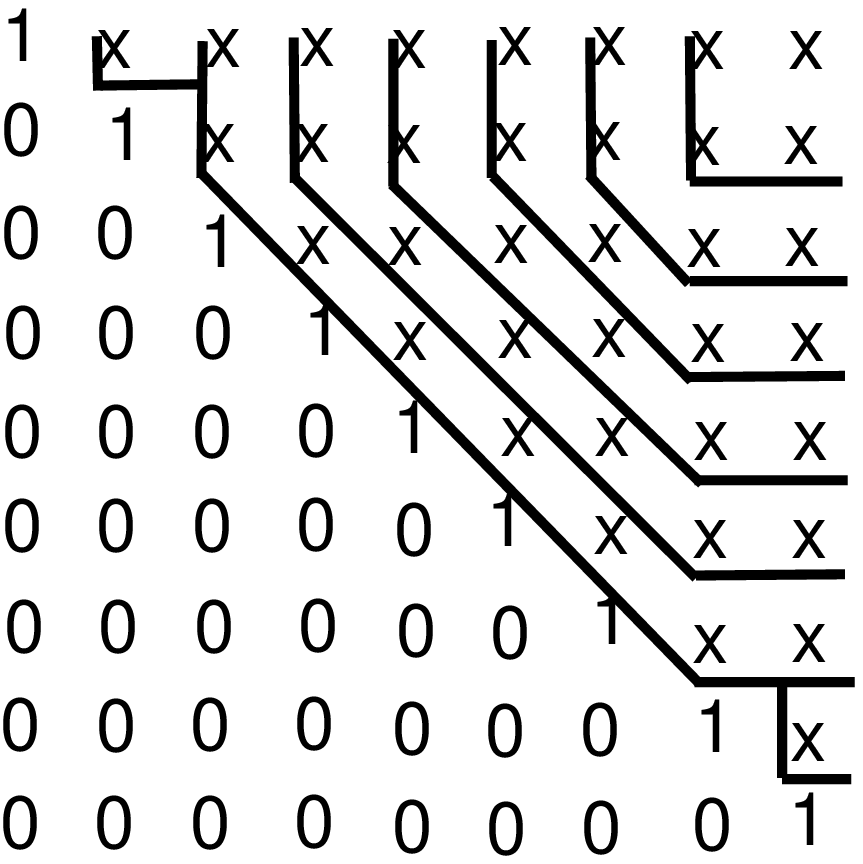} }

	\leavevmode
	\subfigure []
	{ \label{f:subfig-2}
	  \includegraphics[width=4cm]{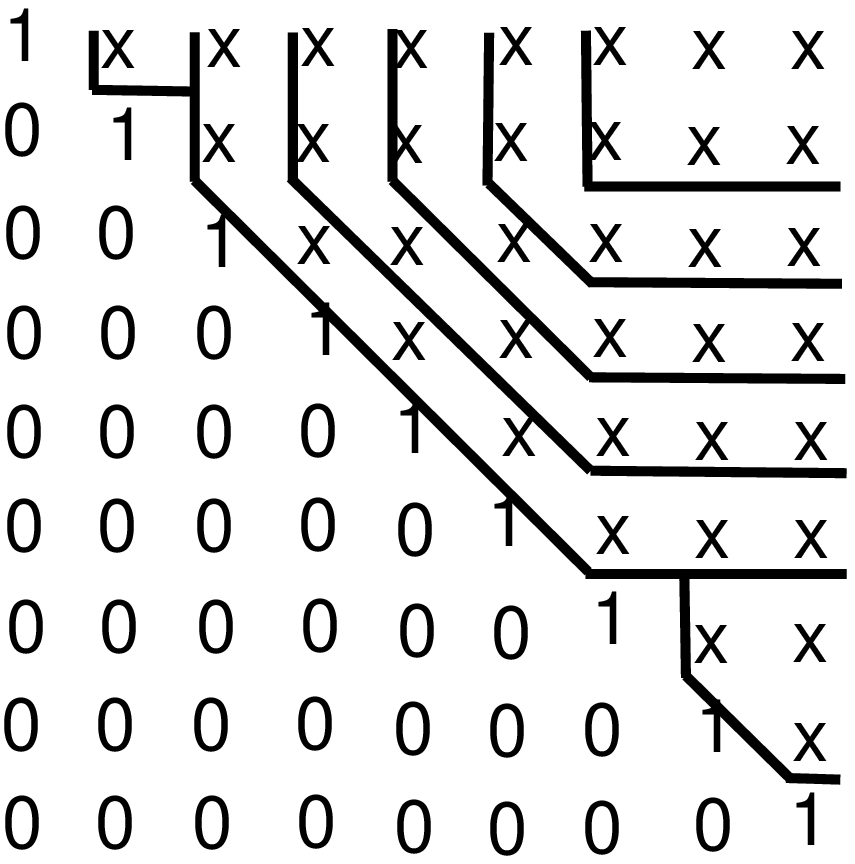} }
}	
\end{center}
\end{figure}

Let $A$ be an invertible $(0,1)$-upper triangular matrix, where the entries in $S_i$ are $i \textrm{ mod } 2$. It follows from the proof of Theorem ~\ref{thm:mainresult} that $A^{-1}$ is an $n \times n$ upper triangular matrix where the diagonal entries are 1, the entries in $S_0$ are 0, and the entries in $S_i$ for $i \geq 1$ are $(-1)^iF_i$. For example when $n=9$, $l=2$;\\
\\
$A=\left(
   \begin{array}{ccccccccc}
     1 & 0 & 1 & 0 & 1 & 0 & 1 & 0 & 0 \\
     0 & 1 & 1 & 0 & 1 & 0 & 1 & 0 & 0 \\
     0 & 0 & 1 & 1 & 0 & 1 & 0 & 1 & 1 \\
     0 & 0 & 0 & 1 & 1 & 0 & 1 & 0 & 0 \\
     0 & 0 & 0 & 0 & 1 & 1 & 0 & 1 & 1 \\
     0 & 0 & 0 & 0 & 0 & 1 & 1 & 0 & 0 \\
     0 & 0 & 0 & 0 & 0 & 0 & 1 & 1 & 1 \\
     0 & 0 & 0 & 0 & 0 & 0 & 0 & 1 & 0 \\
     0 & 0 & 0 & 0 & 0 & 0 & 0 & 0 & 1 \\
   \end{array}
 \right), A^{-1}=\left(
            \begin{array}{ccccccccc}
              1 & 0 & -1& 1 & -2 & 3 & -5 & 8 & 8 \\
              0 & 1 & -1& 1 & -2 & 3 & -5 & 8 & 8 \\
              0 & 0 & 1 & -1 & 1 & -2 & 3 & -5 & -5 \\
              0 & 0 & 0 & 1 & -1 & 1 & -2 & 3 & 3 \\
              0 & 0 & 0 & 0 & 1 & -1 & 1 & -2 & -2 \\
              0 & 0 & 0 & 0 & 0 & 1 & -1 & 1 & 1 \\
              0 & 0 & 0 & 0 & 0 & 0 & 1 & -1 & -1 \\
              0 & 0 & 0 & 0 & 0 & 0 & 0 & 1 & 0 \\
              0 & 0 & 0 & 0 & 0 & 0 & 0 & 0 & 1 \\
            \end{array}
          \right)$\\
\\
and when $n=9$, $l=3$;\\
\\
$A=\left(
  \begin{array}{ccccccccc}
    1 & 0 & 1 & 0 & 1 & 0 & 1 & 1 & 1 \\
    0 & 1 & 1 & 0 & 1 & 0 & 1 & 1 & 1 \\
    0 & 0 & 1 & 1 & 0 & 1 & 0 & 0 & 0 \\
    0 & 0 & 0 & 1 & 1 & 0 & 1 & 1 & 1 \\
    0 & 0 & 0 & 0 & 1 & 1 & 0 & 0 & 0 \\
    0 & 0 & 0 & 0 & 0 & 1 & 1 & 1 & 1 \\
    0 & 0 & 0 & 0 & 0 & 0 & 1 & 0 & 0 \\
    0 & 0 & 0 & 0 & 0 & 0 & 0 & 1 & 0 \\
    0 & 0 & 0 & 0 & 0 & 0 & 0 & 0 & 1 \\
  \end{array}
\right), A^{-1}=\left(
                  \begin{array}{ccccccccc}
                    1 & 0 & -1 & 1 & -2 & 3 & -5 & -5 & -5 \\
                    0 & 1 & -1 & 1 & -2 & 3 & -5 & -5 & -5 \\
                    0 & 0 & 1 & -1 & 1 & -2 & 3 & 3 & 3 \\
                    0 & 0 & 0 & 1 & -1 & 1 & -2 & -2 & -2 \\
                    0 & 0 & 0 & 0 & 1 & -1 & 1 & 1 & 1 \\
                    0 & 0 & 0 & 0 & 0 & 1 & -1 & -1 & -1 \\
                    0 & 0 & 0 & 0 & 0 & 0 & 1 & 0 & 0 \\
                    0 & 0 & 0 & 0 & 0 & 0 & 0 & 1 & 0 \\
                    0 & 0 & 0 & 0 & 0 & 0 & 0 & 0 & 1 \\
                  \end{array}
                \right)\\
$
\\
\\
In general, if $n+l$ is even, $e^TA^{-1}e=2-F_{n-1}$, and hence $A$ is a minimizing extremal matrix (this also includes the case $n=4$). If $n+l$ is odd, $e^TA^{-1}e=2+F_{n-1}$, and hence $A$ is a maximizing extremal matrix. Using these equalities, we obtain the following Fibonacci identities:
\begin{corollary}
$\sum_{i=1}^{n-4} (n-i)(-1)^iF_i+4(-1)^{n-3}F_{n-3}=(-1)^{n-1}F_{n-1}-(n-2)$
\end{corollary}

\begin{corollary}
$\sum_{i=1}^{n-5} (n-i)(-1)^{i}F_i+6(-1)^{n-4}F_{n-4}=(-1)^nF_{n-1}-(n-2)$
\end{corollary}

\section{\textbf{Determinants of (1,2)-matrices}}
In \cite{Bit-Shun} , the following remark, which follows from Cramer’s rule and the multilinearity of the determinant was presented:
\begin{remark}\label{thm:rem}
For any nonsingular matrix $A$,
\begin{center}
    $S(A^{-1})=\frac{det(A+J)-det(A)}{det(A)}$
\end{center}
where $J$ is the matrix whose all entries are 1.
\end{remark}
Recall that it was proved in \cite{Ching_Li} that the maximal determinant of an $n\times n$ Hessenberg (0,1)-matrix is $F_n$. Using our main result and Remark ~\ref{thm:rem}, we obtain another family of matrices whose determinants are strongly related to the Fibonacci sequence.\\
Let $W_n$ be the family of $n \times n$ matrices such that for any $A \in W_n$,
\begin{center}
$A_{ij} =\left\{
    \begin{array}{cl}
      1 & \textrm {if } j>i  \\
      2 & \textrm {if } j=i \\
      1 \textrm { or } 2 & \textrm {if } j<i.

   \end{array} \right.
$ \\
\end{center}

From Remark ~\ref{thm:rem} and Theorem ~\ref{thm:mainresult}, we obtain the following corollary:
\begin{corollary}
Let $n \geq 3$. Then $S=det(A)$, for some $A \in W_n$ iff $S$ is an integer that satisfies
$3-F_{n-1} \leq S \leq 3+F_{n-1}$.
\end{corollary}








\end{document}